\documentclass[11pt]{article}
\usepackage[a4paper]{anysize}\marginsize{3.5cm}{3.5cm}{1.3cm}{2cm}
\pdfpagewidth=\paperwidth \pdfpageheight=\paperheight
\usepackage{amsfonts,amssymb,amsthm,amsmath,amsxtra,amscd,verbatim,eucal}
\usepackage{pgf}
\usepackage[all]{xy}
\usepackage[active]{srcltx}

\theoremstyle{plain}
\newtheorem{thm}{Theorem}[section]
\newtheorem{theorem}[thm]{Theorem}
\newtheorem*{theoremA}{Theorem A}
\newtheorem*{theoremB}{Theorem B}

\newtheorem{lemma}[thm]{Lemma}
\newtheorem{proposition}[thm]{Proposition}

\theoremstyle{definition}
\newtheorem{definition}[thm]{Definition}
\newtheorem{remark}[thm]{Remark}
\newtheorem{example}[thm]{Example}

\newtheorem{question}[thm]{Question}
\newtheorem{conjecture}[thm]{Conjecture}

\renewcommand\phi{\varphi}
\renewcommand\ge{\geqslant}

\renewcommand\le{\leqslant}
\renewcommand\leq{\leqslant}

\newcommand\be{\begin{eqnarray*}}
\newcommand\ee{\end{eqnarray*}}

\newcommand\Q{\mathbb Q}
\newcommand\R{\mathbb R}
\newcommand\C{\mathbb C}
\newcommand\Z{\mathbb Z}
\newcommand\N{\mathbb N}
\renewcommand\P{\mathbb P}

\newcommand\calo{{\mathcal O}}

\newcommand\newop[2]{\def#1{\mathop{\rm #2}\nolimits}}
\newop\log{log}
\newop\clconv{clconv}
\newop\cl{cl}
\newop\ord{ord}
\newop\Gal{Gal}
\newop\SL{SL}
\newop\Bl{Bl}
\newop\mult{mult}
\newop\mass{mass}
\newop\div{div}
\newop\codim{codim}
\newop\inte{int}
\newcommand\eqnref[1]{(\ref{#1})}

\newcommand\D{\Delta}

\newcommand\cF{\mathcal{F}}
\newcommand\rr{\mathrm{rat.rk}}
\newcommand\trdeg{\mathrm{tr.deg}}
\newcommand\nuf{\nu_{\mathrm{flag}}}

\newcommand\fa{\mathfrak{a}}

\newcommand\cO{\mathcal{O}}

\DeclareMathOperator{\vol}{vol}

\newcommand{\ie}{{\rm i.e.\ }}

\newcommand{\e}{\varepsilon}
\newcommand{\la}{\lambda}

\begin{document}

\title{Vanishing sequences and Okounkov bodies}

\author{S\'ebastien Boucksom, Alex K\"uronya, Catriona Maclean, Tomasz Szemberg}
\date{\today}
\maketitle
\thispagestyle{empty}

\begin{abstract} We define and study the vanishing sequence along a real valuation of sections of a line bundle on a normal projective variety. Building on previous work of the first author with Huayi Chen, we prove an equidistribution result for vanishing sequences of large powers of a big line bundle, and study the limit measure; in particular, the latter is described in terms of restricted volumes for divisorial valuations. We also show on an example that the associated concave function on the Okounkov body can be discontinuous at boundary points.
\end{abstract}

%*****************************************************************************

%\tableofcontents

\section{Introduction}
The aim of this paper is to define and study a natural higher dimensional generalization of the classical notion of 'vanishing sequence' in the theory of algebraic curves. Our approach builds on that of \cite{BC}, which studied fairly general filtrations of section rings ; the current paper can be viewed as a detailed study of a special class of filtrations induced by valuations. More general filtrations are in turn closely related to the so-called 'test configurations' in Donaldson's definition of $K$-stability  \cite{Don2,Nys10,Szek}.

\medskip

We work over an algebraically closed field $k$ of arbitrary characteristic. If $L$ is a line bundle on a smooth projective curve $X$ with $H^0(L)\ne 0$, the \emph{vanishing sequence} of $H^0(L)$ at a point $p\in X$ is classically defined as the set
$$
a_{\min}(L,p)=a_1(L,p)<...<a_N(L,p)=a_{\max}(L,p)
$$
of vanishing orders at $p$ of non-zero sections of $L$ (see for instance \cite[p.256]{HM}). The valuation $v:=\ord_p$ defines a decreasing, real filtration
\begin{equation}\label{equ:filt}
\cF^t_v H^0(L):=\left\{s\in H^0(L)\mid v(s)\ge t\right\}\,\,(t\in\R),
\end{equation}
and we then have
\begin{equation}\label{equ:aj}
a_j(L,p)=\inf\left\{t\in\R\mid\codim\cF^t_v H^0(L)\ge j\right\}
\end{equation}
for $j=1,...,N$, and hence $N=h^0(X,L)$.

Using the trivial bound $a_{\max}(mL,p)\le m\deg L$, it is easy to see that the scaled version $(m^{-1}a_j(mL,p))_j$ of the vanishing sequence at $p$ of $H^0(mL)$ equidistributes as $m\to\infty$ with respect to the Lebesgue measure on the interval $[0,\deg L]\subset\R$.

\medskip

If $X$ is now a normal projective variety of arbitrary dimension and $L$ is a line bundle with $H^0(L)\ne 0$, the filtration (\ref{equ:filt}) makes sense for any real valuation $v$ on $X$, and we use (\ref{equ:aj}) to \emph{define} the vanishing sequence
$$
a_{\min}(L,v)=a_1(L,v)\le ...\le a_N(L,v)=a_{\max}(L,v)
$$
of $H^0(L)$ along $v$, again with $N=h^0(L)$. As a set, it coincides with the set of values of $v$ on non-zero sections of $L$, but this time repetitions may occur (unless $v$ has transcendence degree zero, see Lemma \ref{lem:distinct} below).

Assuming that $L$ is big, so that $H^0(mL)\ne 0$ for all $m\gg 1$, a simple subadditivity argument proves the existence of
$$
a_{\min}(\|L\|,v):=\lim_{m\to \infty} m^{-1}a_{\min}(mL,v)\in[0,+\infty)
$$
and
$$
a_{\max}(\|L\|,v):=\lim_{m\to\infty} m^{-1}a_{\max}(mL,v)\in(0,+\infty].
$$
The first invariant $a_{\min}(\|L\|,v)$ coincides by definition with the asymptotic invariant $v(\|L\|)$ as defined in \cite{AIBL}. In particular, it is non-zero iff the center of $v$ on $X$ lies in the non-nef locus $\mathbb{B}_-(L)$.

We say that $v$ has \emph{linear growth} when $a_{\max}(\|L\|,v)$ is finite, \ie when the values of $v$ on sections in $H^0(mL)$ grow at most linearly with $m$. This condition is easily seen to be independent of the choice of the big line bundle $L$ and of $X$ in its birational class; it is thus an intrinsic property of the valuation $v$ on the function field $K$ of $X$.

We prove that any divisorial valuation (and, more generally, any Abhyankar valuation) has linear growth. For a valuation $v$ centered at a closed point, we prove that $v$ has linear growth iff it has volume zero in the sense of \cite{ELS,LM,C2} (cf. Theorem \ref{thm:vol}).

Our first main result describes the asymptotic behavior of the vanishing sequence along $v$ of $H^0(mL)$ as $m\to\infty$.

\begin{theoremA} Let $L$ be a big line bundle on a normal projective variety $X$, and set $N_m:=h^0(mL)$.
\begin{itemize}
\item[(i)] For any real valuation $v$ on $X$, the scaled vanishing sequence
$$
\left(m^{-1} a_j(mL,v)\right)_{1\le j\le N_m}
$$
equidistributes as $m\to\infty$, in the sense that the sequence of discrete probability measures
$$
\nu_k:=\frac{1}{N_m}\sum_j\delta_{m^{-1}a_j(mL,v)}
$$
converges weakly to a positive measure $\mu_{L,v}$ on $\R$.

\item[(ii)] If $v$ has linear growth, then $\mu_{L,v}$ is a probability measure supported on the interval $[a_{\min}(\|L\|,v),a_{\max}(\|L\|,v)]$, and its singular part with respect to the Lebesgue measure consists of at most a Dirac mass at $a_{\max}(\|L\|,v)$.

\item[(iii)] When $v$ doesn't have linear growth we have $\mu_{L,v}=0$.

\end{itemize}
\end{theoremA}

When the base field $k$ has characteristic $0$ and $v$ is Abhyankar, we prove more precisely that $\mu_{L,v}$ is absolutely continuous with respect to the Lebesgue measure, \ie that no Dirac mass can occur at $a_{\max}(\|L\|,v)$. When $v$ is divisorial, we can even provide an explicit formula for the density of $\mu_{L,v}$ in terms of restricted volumes \cite{ELMNP2}, using the differentiability property of the volume function proved in \cite{BFJ,LM}.

\medskip

When $v$ has linear growth on $L$,  Theorem A turns out to be a special case of the main result of \cite{BC}, which also provides a description of the limit measure $\mu_{L,v}$ as the push-forward of the normalized Lebesgue measure on the Okounkov body $\D(L)\subset\R^n$ (with respect to any given flag of subvarieties, see \cite{LM,KK,Bou}) by a concave non-negative usc function
$$
G_{L,v}:\D(L)\to\R_+,
$$
the \emph{concave transform} of $v$ on $\D(L)$. Note that $\D(L)$ and $G_{L,v}$ will depend on the reference flag in general, while the image measure $\mu_{L,v}$ does not. Just like the Okounkov body itself, we prove that $G_{L,v}$ only depends on the numerical equivalence class of $L$. Since it is concave and usc, a simple result from convex analysis guarantees that $G_{L,v}$  is continuous up to the boundary of $\D(L)$ when the latter is a polytope. Our second main result shows that continuity may fail in general when $\D(L)$ has a more erratic boundary:

\begin{theoremB} For an appropriate choice of flag on the blow-up $X$ of $\P^3$ along an adequate smooth curve $C$, there exists an ample prime divisor $D$ on $X$ such that $G_{D,\ord_D}$ is not continuous up to the boundary of $\D(D)$.
\end{theoremB}

\paragraph{Acknowledgments.} We are grateful to Bo Berndtsson, Lawrence Ein, Patrick Graf, Daniel Greb, and Rob Lazarsfeld for helpful discussions.

During this project, S\'ebastien Boucksom was partially supported by the ANR projects MACK and POSITIVE. Alex K\"uronya  was supported in part by  the DFG-Forscher\-grup\-pe 790 ``Classification of Algebraic Surfaces and Compact Complex
Manifolds'', and the OTKA Grants 77476 and  81203 by the Hungarian Academy of Sciences. Tomasz Szemberg's research was partly supported by NCN grant
UMO-2011/01/B/ ST1/04875. Catriona Maclean was supported by  the ANR project CLASS. Part of this work was done while the second author was visiting the Uniwersytet Pedagogiczny in Cracow, and while the second and third
authors were visiting the Universit\'e Pierre et Marie Curie in Paris. We would like to use this
opportunity to thank Anreas H\"oring for the invitation to UPMC and  both institutions  for the excellent working conditions.

\section{Vanishing sequences}
We work over an algebraically closed field $k$, whose characteristic is arbitrary unless otherwise specified. An \emph{algebraic variety} is by definition an integral $k$-scheme of finite type.
\subsection{Real valuations}
We use \cite[Chapter VI]{ZS} and \cite{ELS} as general references on valuations. A \emph{real valuation} $v$ on an algebraic variety $X$ is a valuation on the function field $K$ of $X$, trivial on $k$, with values in the ordered group $(\R,+)$ (\ie of real rank $1$) and admitting a center on $X$. By definition, the latter is a scheme point $c_X(v)\in X$ such that $v\ge 0$ on the local ring at $c_X(v)$ and $v>0$ on its maximal ideal. By the valuative criterion of separatedness, this condition uniquely determines $c_X(v)$, while its existence is guaranteed when $X$ is proper, by the valuative criterion of properness.

The \emph{rational rank} $\rr(v)$ is defined as the maximal number of $\Q$-linearly independent elements in the value group $v(K^*)\subset\R$. The \emph{transcendence degree} $\trdeg(v)$ is defined as the transcendence degree over $k$ of the residue field
$$
k(v):=\{v\ge 0\}/\{v>0\},
$$
and can alternatively be described as the maximal possible dimension of the (closure of the) center of $v$ on a birational model of $X$. The Abhyankar-Zariski inequality states that
$$
\rr(v)+\trdeg(v)\le\dim X,
$$
and an \emph{Abhyankar valuation} is by definition a valuation $v$ for which equality holds.

By the main result of \cite{KnKu},  Abhyankar valuations can be more explicitely characterized as \emph{quasimonomial valuations}, \ie those valuations that become monomial on a birational model of $X$ (see also \cite[Proposition 2.8]{ELS} for a simple proof in characteristic zero). More precisely, $v$ is quasimonomial iff there exists a birational model $X'$ of $X$, proper over $X$ and non-singular at $\xi=c_{X'}(v)$, and a regular system of parameters $(z_1,...,z_r)$ at $\xi'$ (with $r=\rr(v)$, necessarily) such that $v$ is given as a monomial valuation
$$
v\left(\sum_{\alpha\in\N^n}a_\alpha z^\alpha\right)=\min\left\{\sum_i c_i\alpha_i\mid a_\alpha\ne 0\right\}
$$
on the formal completion $\widehat{\mathcal{O}}_{X',\xi}\simeq k(\xi')[[z_1,...,z_r]]$, for some $\Q$-linearly independent weights $c_1,...,c_r\in\R_+$.

In particular, the value group of an Abhyankar valuation is finitely generated (and hence a free abelian group), in stark contrast with more general valuations: according to \cite[p.102]{ZS}, \emph{any} subgroup of $(\Q,+)$ can be realized as the value group of a real valuation on $X=\mathbb{P}^2$.

As an important special case, an Abhyankar valuation $v$ with $\rr(v)=1$ is the same thing as a \emph{divisorial valuation}, \ie a valuation of the form $v=c\,\ord_E$ with $c>0$ and $E\subset X'$ a prime divisor on a birational model $X'$ of $X$, proper over $X$.

At the other end of the spectrum, a valuation $v$ \emph{of maximal rational rank} (\ie such that $\rr(v)=\dim X$) also is an Abhyankar valuation. Its center on every model is then a closed point, \ie $\trdeg(v)=0$, and this property easily implies that
\begin{equation}\label{equ:dim}
\#v\left(E\setminus\{0\}\right)=\dim E
\end{equation}
for every finite dimensional subspace $E\subset K$ (see for instance \cite[Proposition 2.23]{Bou}).

The following simple consequence of the above description of Abhyankar valuations will come in handy later on.
\begin{lemma}\label{lem:dom} If $v$ is an Abhyankar valuation on $X$, then there exists a divisorial valuation $v'$ such that $v\le v'$ on the local ring of $X$ at $c_X(v)$.
\end{lemma}
\begin{proof} As recalled above, there exists a proper birational morphism $\pi:X'\to X$ which is smooth at $\xi':=c_{X'}(v)$ and a regular system of parameters $(z_1,...,z_r)$ at $\xi'$ with respect to which $v$ is monomial. Setting $c_i:=v(z_i)$ we pick rational numbers $c_i'\ge c_i$ and denote by $v'$ the corresponding monomial valuation. Then $v'$ is divisorial since it is Abhyankar with $\rr(v)=1$, and we have $v\le v'$ on $\cO_{X',\xi'}$, hence also on $\cO_{X,c_X(v)}$.
\end{proof}

\subsection{The vanishing sequence along a valuation}
We assume from now on that $X$ is a normal projective variety, and let $v$ be a real valuation on $X$. For each line bundle $L$ on $X$ and each non-zero section $s\in H^0(L)$, we can make sense of $v(s)\in[0,+\infty)$ by trivializing $L$ near the center $c_X(v)$, which identifies $s$ with a local regular function. Since any two local trivializations of $L$ differ by a unit, this is well-defined, and the usual property
$$
v(s+s')\ge\min\left\{v(s),v(s')\right\}
$$
is satisfied for any two sections $s,s'\in H^0(L)$ (with the usual convention that $v(0)=+\infty$). As a consequence, the function
$$
v:H^0(L)\to[0,+\infty]
$$
is uniquely determined by the corresponding (decreasing, real) filtration $\cF_v$ of $H^0(L)$ by linear subspaces, defined by
$$
\mathcal F^t_v H^0(L):=\left\{s\in H^0(L)\mid v(s)\ge t\right\}
$$
for all $t\in\R$.

\begin{definition}\label{def:vanishing sequence} Let $L$ be a line bundle on $X$ such that $N:=h^0(L)$ is non-zero. The \emph{vanishing sequence along $v$ of $H^0(L)$} is the sequence
$$
a_{\min}(L,v)=a_1(L,v)\le...\le a_{N}(L,v)=a_{\max}(L,v)
$$
defined by
$$
a_j(L,v)=\inf\left\{t\in\R\mid\codim\cF^t_v H^0(L)\ge j\right\}
$$
for $j=1,...,N$.
\end{definition}

\begin{remark} In \cite[Definition 1.2]{BC}, the \emph{jumping numbers} of the filtration $\cF_v$ are defined as
$$
e_j=\sup\left\{t\in\R\mid\dim\cF^t_v H^0(L)\le j\right\}.
$$
They relate to the vanishing sequence by $e_j=a_{N-j}$.
\end{remark}
As a set, the vanishing sequence coincides with $v\left(H^0(L)\setminus\{0\}\right)\subset\R_+$, with $a_{\min}(L,v)$ and $a_{\max}(L,v)$ being respectively the smallest and largest value taken by $v$ on a non-zero section of $L$. But there will be repetitions in general, counted in such a way that the basic formula
\begin{equation}\label{equ:der}
-\frac{d}{dt}\dim\cF^t_v H^0(L)=\sum_{j=1}^{N}\delta_{a_j(L,v)}
\end{equation}
holds as distributions on $\R$ (compare \cite[(1.3)]{BC}). We note:

\begin{lemma}\label{lem:distinct} If the real valuation $v$ has transcendence degree $0$ (in particular, if $v$ has maximal rational rank), then the vanishing sequence along $v$ of $H^0(L)$ admits no repetition, \ie $a_i(L,v)<a_j(L,v)$ for $i<j$.
\end{lemma}
\begin{proof} As mentioned above, a valuation $v$ with transcendence degree $0$ satisfies (\ref{equ:dim}) for any finite dimensional linear space $E$ of rational functions, see \cite[Proposition 2.23]{Bou}. In particular, we have
$$
\# v\left(H^0(L)\setminus\{0\}\right)=h^0(L),
$$
which implies that the vanishing sequence along $v$ of $H^0(L)$ admits no repetion.
\end{proof}

Finally, we record the following birational invariance property of vanishing sequences:
\begin{lemma}\label{lem:birinv} If $\pi:X'\to X$ is a birational morphism between normal projective varieties and $L$ is a line bundle on $X$ with $H^0(L)\ne 0$, then we have for each real valuation $v$
$$
a_j(\pi^*L,v)=a_j(L,v)\text{  for  }j=1,...,N.
$$
\end{lemma}
\begin{proof}  We have $\pi_*\cO_{X'}=\cO_X$ since $\pi$ is birational and $X$ is normal, and the projection formula therefore shows that $\pi^*$ induces an isomorphism
$$
\cF^t_v H^0(L)\simeq\cF^t_v H^0(\pi^*L)
$$
for all $t\in\R$.
\end{proof}

\subsection{Linear growth and the volume}
Given any two line bundles $L,L'$ and sections $s,s'$ of $L,L'$ respectively, we plainly have
$$
v(s\otimes s')=v(s)+v(s').
$$
This yields in particular the super and subadditivity properties
$$
a_{\max}((m+m')L,v)\ge a_{\max}(mL,v)+a_{\max}(m'L,v)
$$
and
$$
a_{\min}((m+m')L,v)\le a_{\min}(mL,v)+a_{\min}(m'L,v)
$$
for all $m,m'\in\N$ such that $H^0(mL)$ and $H^0(m'L)$ are non-zero. By the so-called 'Fekete lemma', we infer:

\begin{lemma}\label{lem:subadd} If $L$ is a big line bundle, then $m^{-1}a_{\max}(mL,v)$ and $m^{-1}a_{\min}(mL,v)$ admit limits $a_{\max}(\|L\|,v)\in(0,+\infty]$ and $a_{\min}(\|L\|,v)\in[0,+\infty)$ as $m\to\infty$. In fact, we have
$$
a_{\max}(\|L\|,v)=\sup_{m\ge m_0}m^{-1}a_{\max}(mL,v)
$$
and
$$
a_{\min}(\|L\|,v)=\inf_{m\ge m_0}m^{-1}a_{\min}(mL,v)
$$
for any choice of $m_0\ge 1$ such that $H^0(mL)\ne 0$ for $m\ge m_0$.
\end{lemma}
\begin{remark} Subadditivity of the smallest jumping number can fail for general multiplicative filtrations on the algebra of sections
$$
R(L):=\bigoplus_{m\in\N}H^0(mL),
$$
as considered in \cite{BC}. What is special with $\cF_v$ is the mutiplicativity of the corresponding norm with respect to the trivial valuation of $k$ on the algebra $R(L)$.
\end{remark}

In the notation of \cite[\S 2]{AIBL}, we have
$$
a_{\min}(\|L\|,v)=v(\|L\|)
$$
We thus get
$$
a_{\min}(\|L\|,v)>0\Longrightarrow c_X(v)\in\mathbb{B}_-(L),
$$
where the right-hand side denotes the \emph{restricted base locus} (aka \emph{non-nef locus}). In particular, $a_{\min}(\|L\|,v)$ is always zero when $L$ is nef.

 The converse implication holds at least when $X$ is smooth and $k$ has characteristic $0$, by \cite[Proposition 2.8]{AIBL}.
\begin{lemma}\label{lem:indep} If $a_{\max}(\|L\|,v)$ is finite for a given big line bundle on a given normal projective variety $X$, then $a_{\max}(\|L'\|,v)$ is also finite for any big line bundle on any normal projective variety $X'$ birational to $X$.
\end{lemma}
\begin{proof} By the birational invariance property of Lemma \ref{lem:birinv}, we may assume that $X'=X$. Since $L$ is big, there exists $a\gg 1$ and a non-zero section $\sigma\in H^0(aL-L')$, so that for each $m\in\N$ $H^0(mL')$ injects into $H^0(ka L)$ via $s\mapsto s\otimes\sigma^k$. It follows that
$$
a_{\max}(mL',v)\le a_{\max}(mL',v)+k v(\sigma)\le a_{\max}(ka L)=O(m),
$$
and hence $a_{\max}(\|L'\|,v)<+\infty$.
\end{proof}

We may thus introduce:
\begin{definition}\label{defi:lingrowth} A real valuation $v$ on the function field $K/k$ has \emph{linear growth} if $a_{\max}(\|L\|,v)$ is finite for some (hence any) big line bundle $L$ on some (hence any) normal projective model $X$ of $K$.
\end{definition}

Here is an equivalent formulation:
\begin{proposition}\label{prop:num} A real valuation $v$ on a normal projective variety $X$ has linear growth iff for each big numerical class $\alpha\in N^1(X)_\R$ we have
$$
\sup_{D\equiv\alpha}v(D)<+\infty,
$$
where $D$ ranges over all effective $\R$-Cartier divisors in the classe of $\alpha$ (and the value of $v$ on an $\R$-divisor is defined by linearity).
\end{proposition}
\begin{proof} One direction is clear, since $D:=\frac{1}{m}\div(s)$ is in particular an effective $\Q$-divisor in the numerical class of $L$ for each non-zero section $s\in H^0(mL)$.

Conversely, assume that $v$ has linear growth. Given a big class $\alpha$, we may find a big line bundle $L$ on $X$ such that $c_1(L)-\alpha$ is the class of an effective $\R$-divisor $E$. We then have
$$
v(E)+\sup_{D\equiv\alpha}v(D)\le\sup_{D'\equiv L} v(D'),
$$
and we may thus assume wlog that $\alpha=c_1(L)$ is the numerical class of a big line bundle $L$. By Lemma \ref{lem:conv} below and linearity, we are reduced to proving that $v(D)=O(m)$ for all effective Cartier divisors $D$ numerically equivalent to $mL$.

We can then follow the usual argument relying on \cite[Lemma 2.2.42]{PAG}. The latter yields the existence of a very ample line bundle $A$ such that $A+N$ is very ample for every nef line bundle $N$. Since $mL-D$ is by assumption numerically trivial, it follows in particular that $A+mL-D$ has a non-zero section.

Since $L$ is big, we may assume that $m$ is large enough to guarantee that $H^0(mL-A)\ne 0$. Twisting the canonical section $\sigma_D$ of $\cO_X(D)$ by a non-zero section of $H^0(mL-A)$ and a non-zero section of $A+mL-D$ yields a section $s$ of $mL$, and we thus get as desired
$$
v(D)\le v(s)=O(m).
$$
\end{proof}

\begin{lemma}\label{lem:conv} If $\alpha\in N^1(X)_\Q$ is a rational numerical class, then every effective $\R$-Cartier divisor $D\equiv\alpha$ can be written as a convex combination of effective $\Q$-Cartier divisors in the class of $\alpha$.
\end{lemma}
\begin{proof} If we let $E_i$ be the irreducible components of $D$, then the affine subspace $W$ of $\sum_i\R E_i$ consisting of $\R$-Cartier divisors supported on $\sum_i E_i$ and lying in the numerical class $\alpha$ is rational. As a consequence, $D\in W\cap\sum_i\R_+^* E_i$ can be written as a convex combination of elements in $W_\Q\cap\sum_i\R_+^* E_i$, and the result follows.
\end{proof}

\begin{proposition}\label{prop:linabhy} Every Abhyankar valuation has linear growth.
\end{proposition}
\begin{proof} Let $L$ be a big line bundle on $X$. We start with a simple observation: if $v$ and $v'$ are two real valuations on $X$ such that $v\le v'$ on the local ring of $X$ at $c_X(v)$, then $a_{\max}(\|L\|,v)\le a_{\max}(\|L\|,v')$, so that $v$ has linear growth whenever $v'$ does.

By Lemma \ref{lem:dom}, we are thus reduced to the case of a divisorial valuation $v$. By Lemma \ref{lem:birinv}, we may replace $X$ with a higher birational model and assume that $v=\ord_E$ with $E$ a prime divisor on $X$. Pick an ample line bundle $A$ on $X$. For each non-zero section $s\in H^0(mL)$ we then have
$$
\ord_E(s)\left(\left(A|_E\right)^{\dim X-1}\right)\le m\left(L\cdot A^{\dim X-1}\right),
$$
which shows as desired that $a_{\max}(mL,v)=O(m)$.
\end{proof}

\begin{remark} In characteristic zero, the result is a weak consequence of \cite[Theorem B]{BFJ2}, and can also be deduced from \cite[Theorem A]{ELS}.
\end{remark}

The converse of Proposition \ref{prop:linabhy} fails in general, since already in dimension $2$ there exist non-Abhyankar real valuations $v$ that can be dominated by a divisorial valuation.
\begin{example} Let
$$
\gamma(t):=\sum_{j\ge 0} a_j t^{\beta_j}
$$
be a generalized Puiseux series with $a_j\in k$ and $\beta_j$ an increasing sequence of positive rational numbers bounded above by $C\in [1,+\infty)$ (and hence with unbounded denominators). Then $\gamma$ defines a valuation centered at the origin of $\mathbb{A}^2$ by setting
$$
v(P):=\ord_0 P(t,\gamma(t))
$$
for $P\in k[x,y]$. Using $\beta_j\le C$ for all $j$, it is straightforward to check that $v\le C v_0$ on $k[x,y]$, where $v_0$ is the divisorial valuation on $\mathbb{A}^2$ given by vanishing order at the origin. Since $v_0$ has linear growth, so does $v$. On the other hand, it follows from \cite[Chapter 4]{FJ} that $v$ is not an Abhyankar valuation.
\end{example}

We may however ask:
\begin{conjecture}\label{conj:lin} A real valuation $v$ has linear growth iff there exists a divisorial valuation $v'$ such that $v\le v'$ at the center of $v$ on some birational model.
\end{conjecture}
As we shall see, the conjecture holds at least when $v$ is centered at a closed point on some birational model.

To this end, we will relate the linear growth condition to notion of \emph{volume} of a valuation. Let $v$ be a real valuation $v$ with center $c_X(v)=\xi$ and valuation ideals
$$
\fa_m:=\left\{v\ge m\right\}\subset\cO_X
$$
for $m\in\N$, and set $d=\dim\cO_{X,\xi}$. By \cite{ELS,LM,C2}, the limit
$$
\vol_X(v):=\lim_{m\to+\infty}\frac{d!}{m^d}\,\mathrm{length}\left(\calo_{X,\xi}/\fa_m\right)
$$
exists in $[0,+\infty)$, and is called the \emph{volume} of $v$. It can also expressed in terms of the Hilbert-Samuel multiplicities of the valuation ideals:
$$
\vol_X(v)=\lim_{m\to+\infty}\frac{e\left(\fa_m\right)}{m^d}.
$$

\begin{theorem}\label{thm:vol} For a real valuation $v$ centered at a closed point $x$ of a normal projective variety $X$,  the following conditions are equivalent.
\begin{itemize}
\item[(i)] $v$ has linear growth.
\item[(ii)] $\vol_X(v)>0$.
\item[(iii)] there exists a divisorial valuation $w$ centered at $x$ such that $v\le w$ on $\cO_{X,x}$.
\end{itemize}
\end{theorem}

\begin{proof} (i)$\Longrightarrow$(ii) is elementary. Indeed, pick any big line bundle $L$ on $X$ and a rational $c>a_{\max}(\|L\|,v)$. By definition of the latter, we have
$$
H^0(X,\cO(mL)\otimes\fa_{mc})=0
$$
for all $m\gg 1$, which means that the restriction map
$$
H^0(X,mL)\to\cO_X(mL)\otimes\left(\cO_{X,x}/\fa_{mc}\right)
$$
is injective. Setting $n:=\dim X$, it follows that
$$
\frac{n!}{m^n}h^0(mL)\le\frac{(mc)^n}{n!}\,\mathrm{length}\left(\cO_{X,x}/\fa_{mc}\right),
$$
and hence $\vol(L)\le c^n\vol_X(v)$ in the limit. This implies that $\vol_X(v)>0$, and more precisely
$$
a_{\max}\left(\|L\|,v\right)\ge\left(\frac{\vol(L)}{\vol_X(v)}\right)^{1/n}\in(0,+\infty].
$$
The proof of (ii)$\Longrightarrow$(iii) is more involved. By Lemma \ref{lem:LM}, for each sufficiently large multiple $L$ of a given very ample line bundle $H$, $\cO_X(mL)\otimes\fa_m$ is globally generated for all $m\ge 1$, and further satisfies $H^1(X,\cO_X(mL)\otimes\fa_m)=0$. The latter condition yields the surjectivity of the restriction map
$$
H^0(X,mL)\to\cO_X(mL)\otimes\left(\cO_{X,x}/\fa_m\right),
$$
and hence
$$
h^0(\cO_X(mL)\otimes\fa_m)=h^0(mL)-\dim\left(\cO_{X,x}/\fa_m\right).
$$
If we take $L$ to be a large enough multiple of $H$, we can also achieve that
$$
a_{\max}(\|L\|,v)>1,
$$
simply by homogeneity with respect to $L$. Thanks to \cite[Lemma 1.6]{BC}, this condition implies that the graded algebra
$$
S:=\bigoplus_{m\in\N} H^0\left(X,\cO_X(mL)\otimes\fa_m\right)
$$
contains an ample series, which implies in turn the existence in $(0,+\infty)$ of
$$
\vol(S)=\lim_{m\to\infty}\frac{n!}{m^n} h^0(\cO_X(mL)\otimes\fa_m),
$$
by \cite[Proposition 2.1]{LM}. Note also that $\vol(S)=\vol(L)-\vol_X(v)$ by what we have just seen.

If we assume that $\vol_X(v)>0$, then $\vol(S)<\vol(L)$, and the first author's appendix to \cite{Szek} (which relies on Izumi's theorem) yields a divisorial valuation $w$ such that $w(s)\ge m$ for all $s\in H^0(\cO_X(mL)\otimes\fa_m)$.

Let us now check that $w(f)\ge v(f)$ for each $f\in\cO_{X,x}$. Given $j\ge 1$, define $m_j:=\lfloor j v(f)\rfloor$, so that $f^j$ belongs to $\fa_{m_j}$. Since $\cO_X(m_jL)\otimes\fa_{m_j}$ is globally generated, we can find a section $s_j\in H^0(X,\cO_X(m_jL)\otimes\fa_{m_j})$ such that $s_j=\tau_j f^j$ at $x$ for some local trivialization $\tau_j$ of $m_j L$. We thus get
$$
w(f)= j^{-1}w(s_j)\ge j^{-1} m_j,
$$
hence $w(f)\ge v(f)$ after letting $j\to\infty$.

Finally, (iii)$\Longrightarrow$(i) is a consequence of Proposition \ref{prop:linabhy}.
\end{proof}

The next lemma is a simple variant of \cite[Lemma 3.9]{LM}.
\begin{lemma}\label{lem:LM} Let $H$ be very ample line bundle on a projective variety $X$, and $(\fa_m)_{m\in\N}$ be a graded sequence of non-zero ideals cosupported at a fixed closed point $x$. For each $l\gg 1$, $L:=lH$ satisfies
\begin{itemize}
\item[(i)] $H^q(X,\cO_X(mL)\otimes\fa_m)=0$ for all $m,q\ge 1$;
\item[(ii)] $\cO_X(mL)\otimes\fa_m$ is globally generated for all $m\ge 1$.
\end{itemize}
\end{lemma}
\begin{proof} Property (i) follows directly from \cite[Lemma 3.9]{LM}. In order to get (ii), we rely on the Castelnuovo-Mumford critetion for global generation (cf. \cite[Theorem I.1.8.3]{PAG}), which reduces us to proving the existence of $l_0$ such that
$$
H^q(X,\cO_X(mlH-qH)\otimes\fa_m)=0
$$
for all $q\ge 1$, $l\ge l_0$ and all $m\ge 1$. This vanishing is then checked exactly as in the proof of \cite[Lemma 3.9]{LM}.
\end{proof}

\begin{remark} When $X$ is $2$-dimensional and smooth at $x$, the equivalence between $\vol_X(v)>0$ and linear growth also follows from \cite[Remark 3.3]{FJ}, which gives more precisely that $v\le\vol_X(v)^{-1}\ord_x$ on $\cO_{X,x}$.
\end{remark}

\begin{example}\label{ex:vol} A simple example of a valuation $v$ on $\P^2$ with $\vol_X(v)=0$ (and hence nonlinear growth) is given in \cite[Remark 2.6]{ELS}: let $v$ be the valuation centered at $0\in\mathbb{A}^2$ given by the vanishing order at $t=0$ on the formal arc $t\mapsto (t, e^t-1)$. By \cite[Example 1.4 (iv)]{ELS}, the valuation ideal $\fa_m=\{v\ge m\}$ is generated by $x^m$ and $y-(x+...+x^{m-1}/(m-1)!)$, hence has colength $m$, and it follows that $\vol_X(v)=0$.
\end{example}

\subsection{Equidistribution of vanishing sequences}
In this section we prove our first main result (Theorem A in the introduction), which describes the asymptotic behavior of vanishing sequences.

\begin{thm}\label{thm:equid} Let $L$ be a big line bundle on a normal projective variety $X$, and set $N_m:=h^0(mL)$.
\begin{itemize}
\item[(i)] For any real valuation $v$ on $X$, the scaled vanishing sequence
$$
\left(m^{-1} a_j(mL,v)\right)_{1\le j\le N_m}
$$
equidistributes as $m\to\infty$, in the sense that the sequence of discrete probability measures
$$
\nu_k:=\frac{1}{N_m}\sum_j\delta_{m^{-1}a_j(mL,v)}
$$
converges weakly to a positive measure $\mu_{L,v}$ on $\R$.

\item[(ii)]  If $v$ has linear growth, then $\mu_{L,v}$ is a probability measure with support in $[a_{\min}(\|L\|,v),a_{\max}(\|L\|,v)]$, and whose singular part with respect to the Lebesgue measure consists at most of a Dirac mass at $a_{\max}(\|L\|,v)$.

\item[(iii)] If $v$ doesn't have linear growth, then $\mu_{L,v}=0$.
\end{itemize}
\end{thm}

\begin{remark} When $v$ has linear growth, Theorem \ref{thm:equid} is actually a special case of the main result of \cite{BC}. Indeed, the filtration $\cF_v$ is linearly bounded in the sense of \cite{BC} in that case, and \cite[Theorem 1.11]{BC} directly implies (i), with $\mu_{L,v}$ given as the push-forward of the (normalized) Lebesgue measure on the Okounkov body of $L$ by the concave transform of the filtration (see \S\ref{sec:ok} below for more details). In case (iii) however, \cite{BC} doesn't a priori apply.
\end{remark}

The main ingredient in the proof is:
\begin{lemma}\label{lem:vol} Set $n:=\dim X$. For each $t<a_{\max}(\|L\|,v)$,
$$
\vol(L,v\ge t):=\lim_{m\to\infty}\frac{n!}{m^n}\dim\cF^{mt}_v H^0(mL)
$$
exists in $(0,+\infty)$, and $t\mapsto \vol(L,v\ge t)^{1/n}$ is furthermore concave and non-increasing on $(-\infty,a_{\max}(\|L\|,v))$, and constant on $(-\infty,a_{\min}(\|L\|,v)]$.
\end{lemma}
\begin{proof} If we introduce as in \cite{BC} the graded algebra
$$
R(L,v\ge t):=\bigoplus_{m\in\N}\cF^{ mt}_v H^0(mL),
$$
then $R(L,v\ge t)$ contains an ample series for each $t<a_{\max}(\|L\|,v)$, by \cite[Lemma 1.6]{BC}. The existence of the limit is thus a consequence of \cite[Proposition 2.1]{LM}, which further shows that
$$
\vol(L,v\ge t)=n!\vol\left(\D(L,v\ge t)\right)
$$
where $\D(L,v\ge t)$ denotes the Okounkov body of $R(L,v\ge t)$ with respect to any fixed flag of subvarieties of $X$. Using that
$$
\cF^t_v H^0(mL)\cdot \cF^{t'}H^0(m' L)\subset \cF^{t+t'}_v H^0((m+m')L),
$$
it is easy to check that
$$
(1-\la)\D(L,v\ge t)+\lambda\D(L,v\ge t')\subset \D(L,v\ge (1-\lambda)t+\lambda t')
$$
for all $t,t'\in\R$ and $0\le\lambda\le 1$ (compare \cite[(1.6)]{BC}), and hence
$$
\vol\left(\D(L,v\ge(1-\lambda)t+\lambda t')\right)^{1/n}\ge(1-\la)\vol\left(\D(L,v\ge t)\right)^{1/n}+\la \vol\left(\D(L,v\ge t'\right)^{1/n}
$$
by the Brunn-Minkowski inequality. This shows as desired that $\vol(L,v\ge t)^{1/n}$ is a concave function of $t<a_{\max}(\|L\|,v)$.
\end{proof}

\begin{proof}[Proof of Theorem \ref{thm:equid}] For each $m\in\N$ define $h_m:\R\to\R$ by
$$
h_m(t):=\frac{1}{N_m}\dim\cF^{kt} H^0(mL),
$$
which satisfies
$$
-\frac{d}{dt} h_m=\nu_m
$$
by (\ref{equ:der}). If $h:\R\to\R$ is defined by
$$
h(t)=\frac{\vol(L,v\ge t)}{\vol(L)}
$$
for $t<a_{\max}(\|L\|,v)$ and $h(t)=0$ for $t\ge a_{\max}(\|L\|,v)$, then we get $h_m(t)\to h(t)$ for all $t\ne a_{\max}(\|L\|,v)$, using Lemma \ref{lem:vol} and the fact that $h_m(t)=0$ for all $m$ and all $t>a_{\max}(\|L\|,v)$. Since $0\le h_m\le 1$ is uniformly bounded, the dominated convergence theorem implies that $h_m\to h$ holds in $L^1_\mathrm{loc}$ topology, and hence $-\frac{d}{dt} h_m=\nu_m$ converges weakly on $\R$ to
$$
\mu_{L,v}:=-\frac{d}{dt} h,
$$
which is necessarily a positive measure (as the weak limit of such measures), and is supported on $[a_{\min}(\|L\|,v),a_{\max}(\|L\|,v)]$ since $h$ is constant outside this interval. On $(-\infty,a_{\max}(\|L\|,v))$
$h^{1/n}$ is further concave, hence locally Lipschitz continuous, and it follows that $\mu_{L,v}$ has $L^\infty_{\mathrm{loc}}$ density with respect to Lebesgue measure on $(-\infty,a_{\max}(\|L\|,v))$.

If $v$ has linear growth, then all probability measures $\nu_m$ are supported in the fixed compact set $[0,a_{\max}(\|L,\|,v)]$, and it follows that the weak limit  $\mu_{L,v}$ also is a probability measure.

Otherwise, $a_{\max}(\|L\|,v)$ is infinite, and $h^{1/n}$ is thus concave and non-increasing on $\R$. Since it is also bounded below (by $0$), it has to be constant, which proves that $\mu_{L,v}=0$ is that case.
\end{proof}

For later use, we note:
\begin{lemma}\label{lem:dervol} If $v$ has linear growth, then
$$
\frac{\vol(L,v\ge t)}{\vol(L)}=\mu_{L,v}(x\ge t)
$$
for all $t<a_{\max}(\|L\|,v)$. In particular, $\mu_{L,v}$ has no atom at $a_{\max}(\|L\|,v)$ iff
$$
\lim_{t\to a_{\max}(\|L\|,v)_-}\vol(L,v\ge t)=0.
$$
\end{lemma}

\subsection{The limit measure for Abhyankar valuations in characteristic zero}
In this section we assume that the base field $k$ has characteristic $0$, in order to rely on resolution of singularities.

Recall that the \emph{restricted volume} of a line bundle $L$ on a subvariety $Y\subset X$ is defined as
$$
\vol_{X|Y}(L):=\limsup_{m\to\infty}\frac{d!}{m^d}h^0(X|Y,mL),
$$
where $d:=\dim Y$ and $h^0(X|Y,mL)$ is the rank of the restriction map $H^0(X,mL)\to H^0(Y,mL)$.

\begin{thm}\label{thm:div} Assume that $k$ has characteristic zero. Let $v$ be a divisorial valuation and $L$ be a big line bundle on $X$. By birational invariance and homogeneity, we may assume wlog that $X$ is smooth and $v=\ord_E$ with $E\subset X$ a prime divisor, and we then have
$$
a_{\max}(\|L\|,v)=\sup\left\{t>0\mid L-t E\text{ big}\right\}
$$
and
$$
\mu_{L,v}=\frac{n\vol_{X|E}\left(L-t E\right)}{\vol(L)}dt.
$$
\end{thm}

\begin{proof} In the present case we have $\cF_v^t H^0(mL)\simeq H^0(m(L-tE))$, hence
$$
0<\vol(L,v\ge t)=\vol(L- t E)
$$
for $t<a_{\max}(\|L\|,v)$, which proves that
$$
a_{\max}(\|L\|,v)\le\sup\left\{t>0\mid L-t E\text{ big}\right\}.
$$
Conversely, for each $t>0$ such that $L-t E$ is big we have
$$
\cF_v^t H^0(X,mL)\simeq H^0(m(L- t E)\ne 0
$$
for all $m\gg 1$, which implies that $t\le a_{\max}(\|L\|,v)$.

By continuity of the volume function \cite{PAG}, we thus have
$$
\lim_{t\to a_{\max}(\|L\|,v)} \vol(L,v\ge t)=0,
$$
which proves that $\mu_{L,v}$ is absolutely continuous with respect to Lebesgue measure by Lemma \ref{lem:dervol}. On $(-\infty,a_{\max}(\|L\|,v))$, $\mu_{L,v}$ is the weak derivative of
$$
-\frac{\vol(L,v\ge t)}{\vol(L)}=-\frac{\vol(L- t E)}{\vol(L)}.
$$
The result now follows from the differentiability property of the volume function \cite[Corollary C]{BFJ}, \cite[Corollary C]{LM}.
\end{proof}

For a general Abhyankar valuation we prove:
\begin{proposition} Let $v$ be an Abhyankar valuation, and $L$ be a big line bundle on $X$. Then $\mu_{L,v}$ is absolutely continuous with respect to Lebesgue measure.
\end{proposition}
\begin{proof} Arguing as in the proof of Lemma \ref{lem:dom} shows the existence of a sequence of divisorial valuations $v_j$ such that
$$
v_j\le v\le (1+\e_j)v_j
$$
on the local ring of $X$ at $c_X(v)$, for some sequence $\e_j\to 0$. It follows that
\begin{equation}\label{equ:filtr}
\cF_{v_j}^{ m t} H^0(mL)\subset\cF_v^t H^0(mL)\subset \cF_{v_j}^{m t(1+\e_j)^{-1}} H^0(mL)
\end{equation}
for all $m,j$ and all $t\in\R$. If we define $g_j:\R\to\R$ by $g_j(t)=\vol(L,v_j\ge t)$ if $t<a_{\max}(\|L\|,v_j)$ and $g_j(t)=0$ otherwise, then $g_j$ is continuous on $\R$ by Theorem \ref{thm:div}, and (\ref{equ:filtr}) easily implies that $g^j$ converges uniformly on $\R$ to the function $g:\R\to\R$ defined by $g(t)=\vol(L,v\ge t)$ for $t<a_{\max}(\|L\|,v)$ and $g(t)=0$ for $t\ge a_{\max}(\|L\|,v)$. In particular, $g$ is continuous, and we get as desired
$$
\lim_{t\to a_{\max}(\|L\|,v)}\vol(L,v\ge t)=0.
$$
\end{proof}

\begin{question} Is it true that $\mu_{L,v}$ is absolutely continuous for all real valuations $v$ with linear growth?
\end{question}

\section{The concave transforms of a valuation on the Okounkov body}\label{sec:ok}
\subsection{The concave transform of a filtration}
Let $X$ be a projective variety of dimension $n$. To a flag of subvarieties
$$
X=Y_0\supset Y_1\supset...\supset Y_n=\{p\}
$$
with $\codim Y_i=i$ and such that all $Y_i$'s are smooth at the (closed) point $p\in X$, one attaches a rank $n$ valuation $\nuf:k(X)^*\to\Z^n$ whose components are given by successive vanishing orders along the $Y_i$'s. Until further notice, we fix the choice of such a flag.

Given a graded subalgebra $R$ of the algebra of sections $R(L)$ of a line bundle $L$ on $X$, the \emph{Okounkov body} $\Delta(R)$ of $R$ is defined as the closure in $\R^n$ of
$$
\bigcup_{m\ge 1}\left\{m^{-1}\nuf(s)\mid s\in R_m\setminus\{0\}\right\}.
$$
It is a compact convex subset of $\R^n$, contained in the quadrant $\R_+^n$. We refer to \cite{LM,KK,Bou} for more information on this construction.

Assume now that $L$ is big, so that $\D(L)$ has non-empty interior, \ie is a \emph{bona fide} convex body. For each $m\in\N$, let $(\cF^t H^0(mL))_{t\in\R_+}$ be a decreasing filtration of $H^0(mL)$, and assume that the corresponding filtration $\cF$ of the graded algebra $R(L)$ is multiplicative, in the sense that
$$
\cF^t H^0(mL)\cdot\cF^s H^0(mL)\subset\cF^{t+s}H^0((k+m)L)
$$
for all $s,t\in\R_+$, $k,m\in\N$. For each $t\in\R_+$, one introduces as in \cite{BC} a graded subalgebra $R^t$ of $R(L)$ with graded pieces
$$
R^t_m:=\cF^{m t}H^0(mL).
$$
If $\cF$ is linearly bounded above, \ie if
$$
e_{\max}(\cF):=\sup_{m\ge 1}\left(m^{-1}\sup\left\{t\in\R_+\mid\cF^tH^0(mL)\ne 0\right\}\right)
$$
is finite, then it is shown in \cite{BC} that
$$
\Delta^t(L):=\Delta(R^t)\subset\D(L)
$$
is a convex body for each $t<e_{\max}(\cF)$. The \emph{concave transform} of $\cF$ is the concave usc function $G_\cF:\D(L)\to[0,+\infty)$ defined by
\begin{equation}\label{equ:G}
G_\cF(x):=\sup\left\{t\in\R_+\mid x\in\D^t(L)\right\}.
\end{equation}
By the main result of \cite{BC}, the push-forward by $G_\cF$ of the Lebesgue measure  describes the asymptotic distribution as $m\to\infty$ of the scaled jumping numbers of $\cF^t H^0(mL)$.

Specializing this to the filtration $\cF_v$ induced by a real valuation $v$ with linear growth we set $G_{L,v}:=G_{\cF_v}$, and call it the \emph{concave transform of the valuation $v$}. The limit measure $\mu_{L,v}$ in Theorem \ref{thm:equid} can now be described as the push-forward by $G_{L,v}$ of the Lebesgue measure $\la$ on $\D(L)$, normalized to mass $1$.

Recall from \cite[Proposition 4.1]{LM} that $\D(L)$ only depends on the numerical equivalence class of $L$. We similarly show:

\begin{proposition}\label{prop:num equiv} Let $L$ be a big line bundle on $X$. For any real valuation $v$, the limit measure $\mu_{L,v}$ only depends on the numerical equivalence class of $L$, and the same property holds for the concave transform
$$
G_{L,v} : \Delta(L)\to \R
$$
when $v$ has linear growth on $R(L)$.
\end{proposition}
\begin{proof}
Fix an arbitrary numerically trivial line bundle $P$ on $X$ and set $L':=L+P$. Following the train of thought of the proof of \cite[Proposition 4.1 (i)]{LM}, we will show that
$$
\D(L',v\ge t)=\D(L,v\ge t)
$$
for all $t\in\R_+$, which will yield both results.

By \cite[Lemma 2.2.42]{PAG}, there exists a very ample line bundle $A$ on $X$ such that $A+N$ is very ample for every nef line bundle $N$, and in particular $A+l P$ is very ample for all $l\in\Z$.

Since $L$ is big, we may find $m\gg 1$ and a non-zero section $\sigma\in H^0(mL-A)$. We write
$$
 (k+m)L'=mL+(mL-A)+(A+(k+m)P).
$$
 By very ampleness, for each $k$ we can find a section $\tau_m\in H^0(A+(k+m)P)$ that does not vanish at the center on $X$ of the flag valuation $\nuf$, so that $\nuf(\tau_k)=0$. For each $s\in\cF^{kt}_v H^0(mL)$, setting $s':=s\cdot\sigma\cdot\tau_k$ defines a section in $\cF^{kt}_v H^0((k+m)L')$, and we have
$$
\nuf(s')=\nuf(s)+\nuf(\sigma).
$$
It follows that
$$
\nuf\left(\cF^{kt}H^0(mL)\setminus\{0\}\right)+\nuf(\sigma)\subset\nuf\left(\cF^{kt}H^0((k+m)L)\setminus\{0\}\right),
$$
and hence
$$
\D\left(\bigoplus_{m\in\N}\cF^{kt}H^0(mL)\right)\subset\D\left(\bigoplus_{m\in\N}\cF^{kt}H^0(mL')\right).
$$
In other words, we have proved that $\D(L,v\ge t)\subset\D(L',v\ge t)$, and the result follows by symmetry.
\end{proof}

We now consider three examples where $G_{L,v}$ can be explicitely described.

\begin{example}[Curves]\rm If $X$ is a curve and $L$ is a big (hence ample) line bundle, then the Okounkov body with respect to any point $p\in X$ is the line segment $\D(L)=[0,\deg L]\subset\R$. For $v=\ord_q$ with $q\in X$, it is straightforward to check that the concave transform $G_{L,v}:[0,\deg L]\to\R$ is given by
$$
G_{L,v}(x)=x
$$
when $q=p$, and
$$
G_{L,v}(x)=\deg(L)-x
$$
otherwise.
\end{example}

We next consider a less trivial 2-dimensional example.

\begin{example}[Projective plane]\label{ex:ok f on p2}
   Set $L=\P^2$, $L=\calo(1)$, and consider the flag defined by a point $p$ on a line $\ell$. We then have
$$
\D(L)=\left\{(x,y)\in\R_+^2\mid x+y\le 1\right\}
$$
Let $v=\ord_z$ for a point $z\in X$. One can then check that
$$
G_{L,v}(x,y)=x+y
$$
for $z=p$, and
$$
G_{L,v}(x,y)=1-x
$$
otherwise.
\end{example}

\begin{example}[One-point blow-up of the projective plane]\label{ex:ok f on blow up of p2}
Let now $f:X=\Bl_q\P^2\to\P^2$ be the blow up of the projective plane
   in a point $q$, with exceptional divisor $F$. Let $p\in\ell\subset X$ be the flag given by taking the strict transform of a point on a line not passing through $q$.  We work with a $\Q$-divisor $L_{\lambda}=f^*H-\lambda F$ with $H=\mathcal{O}(1)$, for some fixed $\lambda\in\Q\cap[0,1]$. A direct computation using \cite[Theorem 6.2]{LM} gives that the Okounkov body of $L_\lambda$ has the shape
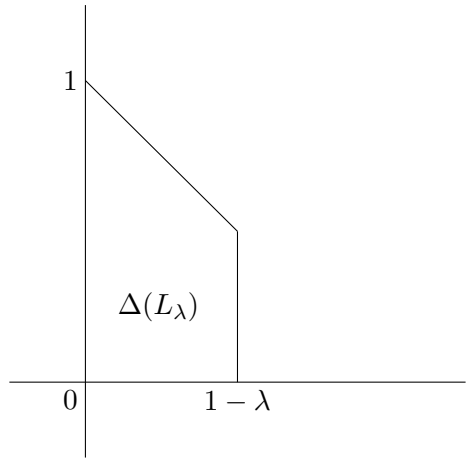
\begin{center}
\begin{pgfpicture}{0cm}{0cm}{6cm}{6cm}
   \pgfline{\pgfxy(0,1)}{\pgfxy(6,1)}
   \pgfline{\pgfxy(1,0)}{\pgfxy(1,6)}
   \pgfline{\pgfxy(1,5)}{\pgfxy(3,3)}
   \pgfline{\pgfxy(3,1)}{\pgfxy(3,3)}
   \pgfputat{\pgfxy(3,0.9)}{\pgfbox[center,top]{$1-\lambda$}}
   \pgfputat{\pgfxy(0.9,0.9)}{\pgfbox[right,top]{$0$}}
   \pgfputat{\pgfxy(0.9,5)}{\pgfbox[right,center]{$1$}}
   \pgfputat{\pgfxy(2.5,2)}{\pgfbox[right,center]{$\D(L_{\lambda})$}}
\end{pgfpicture}
\end{center}
   Consider the divisorial valuation $v=\ord_z$ attached to a point $z\in X$. For $z=p$, we find as before $G_{L,v}(x,y)=x+y$. Assume now that $z$ is a point
   not on the exceptional divisor $F$ (hence $z$ can be considered also
   as a point on $\P^2$) and not on the line through $p$ and $q$. We have now
   for $(x,y)\in\Delta(L_{\lambda})$
   $$
   G_{L,v}(x,y)=\left\{\begin{array}{ccc}
      1-x & \mbox{ for } & x+y\leq 1-\lambda\\
      2-2x-y-\lambda & \mbox{ for } & 1-\lambda\leq x+y\leq 1
      \end{array}\right.$$
  To see this, we may assume by continuity that $x,y\in\Q$. By construction, $G_{L,v}(x,y)$ is then the maximal vanishing order at $z$ of all effective $\Q$-divisors $D$ on $\P^2$ with $D\sim_{\Q}\mathcal{O}(1)$ and vanishing
   \begin{itemize}
      \item[a)] along $\ell$ to order $x$;
      \item[b)] in $q$ to order $\lambda$;
      \item[c)] in $p$ to order $x$ after dividing by the
      equation of $\ell$ in power $x$ and after restricting to $\ell$.
   \end{itemize}
   Condition a) ''costs'' $xH$, so we are left with
   $(1-x)H-\lambda F$ to take care of conditions b) and c).
   If $y\leq 1-x-\lambda$, then we take a line through the
   points $z$ and $q$ with multiplicity $\lambda$
   and the line through $z$ and $p$ with multiplicity $1-x-\lambda$.
   Their union has multiplicity $\lambda+(1-x-\lambda)=1-x$ at $q$
   and satisfies b) and c). Moreover, there is no $\Q$-divisor
   equivalent to $(1-x)H-\lambda F$ with
   higher multiplicity at $z$, which follows easily from
   B\'ezout's theorem intersecting with both lines.

   The argument in the remaining case $y>1-x-\lambda$ is similar.
   We want to split the divisor so that it produces a high
   vanishing order towards condition c) first and then,
   after arriving to the threshold
   \begin{equation}\label{eq:threshold}
      y'=1-x'-\lambda',
   \end{equation}
   we take again the union of two lines as above. Thus,
   we start with the conic through $q$ and $z$
   tangent to $\ell$ at $p$. We take this conic with
   multiplicity $\alpha$ subject to condition that
   $$y-2\alpha=1-x-2\alpha-(\lambda-\alpha),$$
   which means that the divisor $(1-x-2\alpha)H-(\lambda-\alpha)F$
   satisfies \eqnref{eq:threshold} with $y'=y-2\alpha$,
   $x'=x+2\alpha$ and $\lambda'=\lambda-\alpha$.
   The constructed $\Q$-divisor, consisting of the conic and two lines
   has then multiplicity
   $$x+y+\lambda-1+(1-x-2(x+y+\lambda-1))=2-2x-y-\lambda.$$
   B\'ezout's theorem shows then that there is no divisor of higher
   multiplicity.

\end{example}

\subsection{Continuity of concave transforms on Okounkov bodies}
We start by relating the continuity of concave transforms to the geometry of Okounkov bodies. Let $\D$ be a convex body in $\R^n$. The \emph{extremal function} of $\D$ at a point $p\in\D$ is the concave usc function $E_{\D,p}:\D\to[0,1]$ defined by
$$
E_{\D,p}(x)=\sup\left\{t\in[0,1]\mid x\in tp+(1-t)\D\right\}.
$$
It is elementary to check that the following properties are equivalent (see \cite[Proposition 3]{Howe}):
\begin{itemize}
\item[(i)] $\D$ is conical at $p$, in the sense that $\D$ coincides in a neighborhood of $p$ with a closed convex cone with apex $p$;
\item[(ii)] $E_{\D,p}$ is continuous at $p$;
\item[(iii)] every bounded concave usc function on $\D$ is continuous at $p$.
\end{itemize}
Further, $\D$ is conical at each of its (boundary) points iff it is a polytope. In particular, every concave usc function on a polytope is continous up to the boundary.

\begin{example} Since Okounkov bodies on surfaces are polygones \cite[Theorem B]{KLM}, all concave transforms on Okounkov bodies in dimension two are continuous.

In a similar vein, if $X$ is  a normal projective variety of arbitrary dimension and $L$ is a big line bundle with a finitely generated section ring $R(L)$, then by \cite[Theorem 1]{AKL12} the flag of subvarieties of $X$ can be chosen in such a way that
$\Delta(L)$ is a (rational) simplex. As a consequence, any concave transform on $\D(L)$ is again continuous.
\end{example}

\begin{lemma}[A non-continuity criterion]\label{lem:notcont}
Let $D\subset X$ be a big prime divisor. Then $G_{D,\ord_D}$ coincides with the extremal function of $\D(D)$ at $p=\nuf(D)$. In particular, $\D(D)$ is conical at $p$ iff $G_{D,\ord_D}$ is continuous at $p$.
\end{lemma}
\begin{proof} For all $t\in\R_+$ and $m\in\N$, we have
$$
H^0\left(\left(k-\lceil t\rceil\right)D\right)\simeq\cF^t_{\ord_D} H^0(kD).
$$
It follows easily that $a_{\max}(\|D\|,\ord_D)=1$ and
$$
\Delta^t(D)= tp+ (1-t)\D(D)
$$
for $t\in[0,1]$, hence the result.
\end{proof}

We will also use the following result, which is a consequence of \cite[Proposition 4.10]{Bou}:
\begin{lemma}\label{lem:oknef} Let $X$ be a normal projective variety of dimension $n$ and
$$
X=Y_0\supset ...\supset Y_n=\{p\}
$$
be a flag of subvarieties with $Y_{i+1}$ Cartier in $Y_i$ and such that each $Y_i$ with $i\ge 1$ has the property that every effective divisor on $Y_i$ is nef (this condition being automatic for $i=n-1$ and $n$). Let also $L$ be an \emph{ample} line bundle on $X$. Then $\D(L)$ and
$$
\bigcap_{i\ge 1}\left\{(x_1,...,x_n)\in\R_+^n\mid L|_{Y_i}-x_1Y_1|_{Y_i}-...-x_{i+1}Y_{i+1}\in\mathrm{Nef}(Y_i)\right\}
$$
coincide in the half-space $\{x_1\le a\}$ for $0<a\ll 1$.
\end{lemma}
When $n=3$, the assumption reduces to the fact that the surface $Y_1$ contains no curve with negative self-intersection, and the nef cone of $Y_1$ is then (the closure of) one of the two connected components of the positive cone of the intersection form. If we require
\begin{itemize}
\item[(i)] $(Y_2^2)_{Y_1}>0$
\end{itemize}
then a numerical class $\alpha\in N^1(Y_1)$ is nef iff $(\alpha^2)_{Y_1}\ge 0$ and $(\alpha\cdot Y_2)_{Y_1}\ge 0$, and Lemma \ref{lem:oknef} therefore shows that $\D(L)$ coincides near the plane $(x_1=0)$ with the intersection of the quadrant $\R_+^3$, of the solid quadric
$$
q(x_1,x_2,x_3):=\left(L|_{Y_1}-x_1Y_1|_{Y_1}-x_2Y_2\right)_{Y_1}^2\ge 0
$$
and of the half-space
$$
x_1(Y_1\cdot Y_2)_X+x_2(Y_2^2)_{Y_1}+x_3\le(L\cdot Y_2)_X.
$$
If we further assume that the divisor class $Z:=L|_{Y_1}-Y_2$ on $Y_1$ satisfies
\begin{itemize}
\item[(ii)] $(Z^2)_{Y_1}=0$ and
\item[(iii)] $(Z\cdot Y_2)_{Y_1}>0$,
\end{itemize}
then $p=(0,1,0)$ lies in the interior of the above half-space, so that $\D(L)$ locally coincides near $p$ with $\R_+^3\cap\{q\ge 0\}$. Since we also have $q(p)=0$ and $\frac{\partial q}{\partial x_2}(p)=-2(Z\cdot Y_2)_{Y_1}$ is non-zero, we conclude that $\D(L)$ is not conical at $p$.

It remains to construct an example satisfying (i)--(iii) above and such that $L$ can be represented by a prime divisor $D\ne Y_1$ with
\begin{itemize}
\item[(iv)] $\ord_{Y_2}\left(D|_{Y_1}\right)=1$.
 \end{itemize}
Indeed, we can then set $Y_3$ to be any point of $Y_2$ not on $D|_{Y_1}$ to get $\nuf(D)=(0,1,0)=p$, and it will follow from the above discussion and Lemma \ref{lem:notcont} that $G_{D,\ord_D}$ is not continuous at $p$.

To guarantee that $D$ is prime, we will rely on the following simple criterion:
\begin{lemma}\label{lem:irred} Let $Y\subset X$ be smooth projective varieties such that the restriction map $N^1(X)\to N^1(Y)$ is injective. Let $D$ be an effective divisor on $X$ that doesn't contain $Y$ in its support and such that $D|_{Y}=E_1+E_2$ with $E_1,E_2$ prime divisors and $[E_1]\in N^1(Y)$ not in the image of $N^1(X)$. Then $D$ is a prime divisor.
\end{lemma}
Note that the injectivity of $N^1(X)\to N^1(Y)$ is automatic if $Y$ is an ample divisor and $\dim X\ge 3$, by the Lefschetz hyperplane theorem.
\begin{proof} Assume by contradiction that $D=D'+D''$ with $D',D''$ non-zero effective divisors on $X$. Then $[D'],[D'']\in N^1(X)$ are non-zero since $X$ is projective, and $D'|_{Y},D''|_{Y}$ are non-zero as well by assumption. Since $E_1+E_2=D'|_{Y}+D''|_Y$, $E_1$ must coincide with $D'|_{Y}$, say, which contradicts the fact that $[E_1]$ is not the restriction of a class from $X$.
\end{proof}

\begin{example}[Proof of Theorem B] We work over $k=\C$. As in \cite{KLM}, we can use \cite{C} to obtain the existence of a smooth quartic
surface $S\subseteq\mathbb{P}^3$ without $(-2)$-curve (and hence such that every effective divisor is nef) and of two smooth irreducible curves $C,C'\subset S$ such that $C,C'$ and $\calo_S(1)$ generate the N\'eron-Severi group of $S$. Let $\pi:X\to\P^3$ be the blow-up along $C$, with exceptional divisor $E$, and denote by $Y_1$ the strict transform of $S$, so that $\pi$ induces an isomorphism $Y_1\simeq S$ under which $Y_1\cap E$ corrresponds to $C$. Since $N^1(X)$ is generated by $[E]$ and $[\pi^*\calo(1)]$, it follows that $N^1(X)\to N^1(Y_1)$ is injective, and that $[C']$ (viewed as a class on $Y_1$) is not in the image of $N^1(X)$.

Now let $L$ be any ample line bundle on $X$, ample enough to ensure that $L|_{Y_1}-C'$ is very ample and $H^0(L)\to H^0(Y_1,L|_{Y_1})$ is surjective. We can then choose a smooth irreducible curve $Y_2\in|L|_{Y_1}-C'|$ and an effective divisor $D\in|L|$ such that $D|_{Y_1}=C'+Y_2$. By Lemma \ref{lem:irred}, it follows that $D$ is prime, and (i)--(iv) are satisfied. This concludes the proof of Theorem B.

\end{example}

%********************************************************************************************************************************************

%***************************************************************************** % Addresses

\bigskip \small

\bigskip\noindent
S{\'e}bastien Boucksom,
CNRS--Universit{\'e} Paris 6, Institut de Math{\'e}matiques, F-75251 Paris Cedex 05, France

\nopagebreak\noindent
\textit{E-mail address:} \texttt{boucksom@math.jussieu.fr}

\bigskip\noindent
   Alex K\"uronya,
   Budapest University of Technology and Economics,
   Mathematical Institute, Department of Algebra,
   Pf. 91, H-1521 Budapest, Hungary.

\nopagebreak\noindent
   \textit{E-mail address:} \texttt{alex.kuronya@math.bme.hu}

\medskip\noindent
   \textit{Current address:}
   Alex K\"uronya,
   Albert-Ludwigs-Universit\"at Freiburg,
   Mathematisches Institut,
   Eckerstra{\ss}e 1,
   D-79104 Freiburg,
   Germany.

\bigskip\noindent
   Catriona Maclean,
   Institut Fourier, CNRS UMR 5582   Universit\'e de Grenoble,
   100 rue des Maths,
   F-38402 Saint-Martin d'H\'eres cedex,  France

\nopagebreak\noindent
   \textit{E-mail address:} \texttt{catriona.maclean@ujf-grenoble.fr}

\bigskip\noindent
   Tomasz Szemberg,
   Instytut Matematyki UP,
   Podchor\c a\.zych 2,
   PL-30-084 Kram\'ow, Poland.

\nopagebreak\noindent
   \textit{E-mail address:} \texttt{tomasz.szemberg@uni-due.de}

%*****************************************************************************

\end{document}